\documentclass[final]{siamart1116}

\usepackage{braket,amsfonts}

\usepackage{array}

\usepackage[caption=false]{subfig}

\usepackage{pgfplots}

\newsiamthm{claim}{Claim}
\newsiamremark{rem}{Remark}
\newsiamremark{expl}{Example}
\newsiamremark{hypothesis}{Hypothesis}
\crefname{hypothesis}{Hypothesis}{Hypotheses}

\usepackage{algorithmic}

\usepackage{graphicx,epstopdf}

\Crefname{ALC@unique}{Line}{Lines}

\usepackage{amsopn}

\numberwithin{theorem}{section}

\usepackage{xspace}
\usepackage{bold-extra}
\usepackage[most]{tcolorbox}

\colorlet{texcscolor}{blue!50!black}
\colorlet{texemcolor}{red!70!black}
\colorlet{texpreamble}{red!70!black}
\colorlet{codebackground}{black!25!white!25}

\lstdefinestyle{siamlatex}{%
  style=tcblatex,
  texcsstyle=*\color{texcscolor},
  texcsstyle=[2]\color{texemcolor},
  keywordstyle=[2]\color{texemcolor},
  moretexcs={cref,Cref,maketitle,mathcal,text,headers,email,url},
}

\tcbset{%
  colframe=black!75!white!75,
  coltitle=white,
  colback=codebackground, 
  colbacklower=white, 
  fonttitle=\bfseries,
  arc=0pt,outer arc=0pt,
  top=1pt,bottom=1pt,left=1mm,right=1mm,middle=1mm,boxsep=1mm,
  leftrule=0.3mm,rightrule=0.3mm,toprule=0.3mm,bottomrule=0.3mm,
  listing options={style=siamlatex}
}

\newtcblisting[use counter=example]{example}[2][]{%
  title={Example~\thetcbcounter: #2},#1}

\newtcbinputlisting[use counter=example]{\examplefile}[3][]{%
  title={Example~\thetcbcounter: #2},listing file={#3},#1}

\DeclareTotalTCBox{\code}{ v O{} }
{ 
  fontupper=\ttfamily\color{black},
  nobeforeafter,
  tcbox raise base,
  colback=codebackground,colframe=white,
  top=0pt,bottom=0pt,left=0mm,right=0mm,
  leftrule=0pt,rightrule=0pt,toprule=0mm,bottomrule=0mm,
  boxsep=0.5mm,
  #2}{#1}

\patchcmd\newpage{\vfil}{}{}{}
\usepackage{mathtools}
\newcommand{\R}{\mathbb{R}}
\DeclarePairedDelimiter{\norm}{\lVert}{\rVert}

\begin{tcbverbatimwrite}{tmp_\jobname_header.tex}
\title{Radial Subgradient Method}

\author{Benjamin Grimmer%
  \thanks{Cornell University, Ithaca, NY (\email{bdg79@cornell.edu}).}%
}

\headers{Radial Subgradient Method}
{Benjamin Grimmer}
\end{tcbverbatimwrite}
\input{tmp_\jobname_header.tex}

\ifpdf
\hypersetup{ pdftitle={Radial Subgradient Method} }
\fi

\begin{document}
\maketitle

\begin{tcbverbatimwrite}{tmp_\jobname_abstract.tex}
\begin{abstract}
We present a subgradient method for minimizing non-smooth, non-Lipschitz convex optimization problems. The only structure assumed is that a strictly feasible point is known. We extend the work of Renegar \cite{Renegar2016} by taking a different perspective, leading to an algorithm which is conceptually more natural, has notably improved convergence rates, and for which the analysis is surprisingly simple. At each iteration, the algorithm takes a subgradient step and then performs a line search to move radially towards (or away from) the known feasible point. Our convergence results have striking similarities to those of traditional methods that require Lipschitz continuity. Costly orthogonal projections typical of subgradient methods are entirely avoided.
\end{abstract}


\end{tcbverbatimwrite}
\input{tmp_\jobname_abstract.tex}

\section{Introduction} \label{sec:Introduction}

We consider the convex optimization problem of minimizing a lower-semicontinuous convex function $f:\R^n \to \R\cup\{\infty\}$ when a point $x_0$ in the interior of the domain of $f$ is known. Importantly, it is not assumed that $f$ is either smooth or Lipschitz.
Note that any constrained convex optimization problem fits under this model by setting the function value of all infeasible points to infinity (provided the feasible region is closed and convex, and a point exists in the interior of both the feasible region and the domain of $f$). This model can easily be extended to allow affine constraints by considering the relative interior of the domain instead.

Without loss of generality, we have $x_0=\vec 0\in\text{int dom} f$ and $f(\vec 0) < 0$, which can be achieved by replacing $f(x)$ by $f(x+x_0)-f(x_0)-h$ (for any positive constant $h>0$). Let $f^* = \inf\{f(x)\}$ denote the function's minimum value (which equals $-\infty$ if the function is unbounded below). 

This general minimization problem has been the focus of a recent paper by Renegar \cite{Renegar2016}. Renegar develops a framework for converting the original non-Lipschitz problem into an equivalent Lipschitz problem, in a slightly lifted space. This transformation is geometric in nature and applied to a conic reformulation of the original problem.
By applying a subgradient approach to the reformulated problem, Renegar achieves convergence bounds analogous to those of traditional methods that assume Lipschitz continuity. 

One difference in Renegar's approach is that it guarantees relative accuracy of a solution rather than absolute accuracy. A point $x$ has absolute accuracy of $\epsilon$ if $f(x) - f^* < \epsilon$. In contrast, a point $x$ has relative accuracy of $\epsilon$ if $(f(x) - f^*)/(0-f^*) < \epsilon$. Note that here we measure error relative to an objective value of $0$ since we assume $f(x_0)<0$. This relative accuracy can be interpreted as a multiplicative accuracy through simple algebraic rearrangement:
$$ \frac{f(x) - f^*}{0-f^*} < \epsilon \Longleftrightarrow f(x) < f^* (1-\epsilon).$$

We take a different perspective on the transformation underlying Renegar's framework. Instead of first converting the convex problem into conic form in a lifted space, we define an equivalent transformation directly in terms of the original function.

Before stating our function-oriented transformation, we define the {\it perspective function} of $f$ for any $\gamma>0$ as $f^p(x, \gamma) = \gamma f(x/\gamma)$. Note that perspective functions have been studied in a variety of other contexts. Two interesting recent works have considered optimization problems where perspective functions occur naturally. In \cite{Combettes2016}, Combettes and M{\"u}ller investigate properties of the proximity operator of a perspective function. In \cite{Aravkin2017}, Aravkin et al.\ 
construct a duality framework based on gauges, which extends to perspective functions.

Based on the perspective function of $f$, our functional version of Renegar's transformation is defined as follows.

\begin{definition} \label{def:Radial-Reformulation}
The {\it radial reformulation} of $f$ of level $z\in\R$ is given by
$$\gamma_z(x) = \inf\left\{\gamma > 0 \mid f^p(x,\gamma) \leq z\right\}.$$
\end{definition}
As a simple consequence of Renegar's transformation converting non-Lipschitz problems into Lipschitz ones, the radial reformulation of any convex function is both convex and Lipschitz (shown in Lemma \ref{lem:Lambda-To-Radial-Reformulation} and Proposition \ref{prop:Lipschitz-Of-Radial-Reformulation}). As a result, the radial reformulation is amenable to the application of traditional subgradient descent methods, even if $f$ is not.

We present modified versions of Renegar's algorithms which are both conceptually simpler and achieve notably improved convergence bounds. We elaborate on the relationship between these algorithms in Section \ref{subsec:Renegars-Results}. Let $\partial \gamma_z(x)$ denote the set of subgradients of $\gamma_z(\cdot)$ at $x$.
Then our algorithm is stated in Algorithm \ref{alg:Radial-Subgradient-Method} with step sizes given by a positive sequence $\{\alpha_i\}$.

\begin{algorithm}
\caption{Radial Subgradient Method}
\label{alg:Radial-Subgradient-Method}
\begin{algorithmic}[1]
\STATE $x_0=\vec 0$, $z_{0}=f(x_0)$
\FOR{$i = 0, 1, \dots$}
   \STATE Select $\zeta_i \in \partial \gamma_{z_{i}}(x_i)$ \hfill \COMMENT{Pick a subgradient of $\gamma_{z_{i}}(\cdot)$}
   \STATE $\tilde x_{i+1} = x_i - \alpha_i \zeta_i$ \hfill \COMMENT{Move in the subgradient direction}
   \STATE {\bf if } $\gamma_{z_{i}}(\tilde x_{i+1}) =0$ {\bf then } report unbounded objective and terminate.
   \STATE $(x_{i+1},z_{i+1}) = \dfrac{1}{\gamma_{z_{i}}(\tilde x_{i+1})}(\tilde x_{i+1},z_{i})$
 \hfill \COMMENT{Radially update current solution}
\ENDFOR
\end{algorithmic}
\end{algorithm}

Critical to the strength of our results is the specification of the initial iterate, $x_0 = \vec 0$.
Each iteration of this algorithm takes a subgradient step (with respect to a radial reformulation $\gamma_{z_{i}}(\cdot)$ of changing level) and then moves the resulting point radially towards (or away from) the origin. This radial movement ensures that each iterate $x_i$ lies in the domain of $f$.

Notably, this algorithm completely avoids computing orthogonal projections. Renegar's algorithms also have this computational advantage.
Evaluating $\gamma_z(\cdot)$ only requires a line search, which can be viewed as a radial projection. Further, the subgradients of $\gamma_z(\cdot)$ can be easily computed from normal vectors of the epigraph of $f$ (shown in Proposition \ref{prop:Subgradients-Of-Radial}, which follows directly from an equivalent characterization given in \cite{Renegar2016}). This can lead to substantial improvements in runtime over projected subgradient methods, which require costly orthogonal projections every iteration.
We defer a deeper discussion of this improvement in iteration cost to \cite{Renegar2016}.

We first give a general convergence guarantee for the Radial Subgradient Method, where $\mathrm{dist}(x_0, X)$ denotes the minimum distance from $x_0=\vec 0$ to a set $X$. As a consequence, we find that proper selection of the step size will guarantee a subsequence of the iterates has objective values converging to the optimal value.

Let $\langle\cdot,\cdot\rangle$ denote the inner product on $\R^n$ with associated norm $\norm{\cdot}$.
Define $R = \sup\{r\in\R \mid f(x)\leq 0 \text{ for all } x \text{ with } \norm{x}\leq r\}$. This scalar only appears in convergence bounds, but is never assumed to be known.

\begin{theorem} \label{thm:Convergence-In-General}
Let $\{x_i\}$ be the sequence generated by Algorithm \ref{alg:Radial-Subgradient-Method} with steps sizes $\alpha_i$. Consider any $\hat{f}<0$ with nonempty level set $\hat{X} = \{ x \mid f(x) = \hat{f} \}$. Then for any $k\geq 0$, some iteration $i\leq k$ either identifies that $f$ is unbounded on the ray $\{t \tilde x_{i+1} \mid t>0\}$ or has
\begin{equation*}
\frac{f(x_i) - \hat{f}}{0 -f(x_i)} \leq \frac{\mathrm{dist}(x_0,\hat{X})^2 + \frac{1}{R^2}\sum^k_{j=0}\alpha_j^2\left(\frac{\hat{f}}{z_{j}}\right)^2}{2\sum^k_{j=0}\alpha_j\frac{\hat{f}}{z_{j}}}.
\end{equation*}
\end{theorem}
\begin{corollary} \label{cor:Convergence-Square-Summable}
Consider any positive sequence $\{\beta_i\}$ with $\sum^\infty_{i=0}\beta_i = \infty$ and $\sum^\infty_{i=0}\beta_i^2 < \infty$. 
Let $\{x_i\}$ be the sequence generated by Algorithm \ref{alg:Radial-Subgradient-Method} with $\alpha_i = -z_{i}\beta_i$. Then either some iteration identifies that $f$ is unbounded on the ray $\{t \tilde x_{i+1} \mid t>0\}$ or
\begin{equation*}
\lim_{k\to\infty}\min_{i\leq k}\left\{f(x_i)\right\} = f^*.
\end{equation*}
\end{corollary}

Although the choice of step size in Corollary \ref{cor:Convergence-Square-Summable} guarantees that our algorithm converges, it does not provide any bounds on the rate of convergence.
We bound the convergence rate of the Radial Subgradient Method in two cases: when a target accuracy $\epsilon>0$ is given and when the optimal objective value $f^*$ is given. The resulting convergence bounds for two particular choices of step sizes are stated in the following theorems. 

\begin{theorem} \label{thm:Convergence-When-Optimal-Unknown}
Consider any $\epsilon>0$ and $\hat{f}<0$ with nonempty level set $\hat{X} = \{ x \mid f(x) = \hat{f} \}$. Let $\{x_i\}$ be the sequence generated by Algorithm \ref{alg:Radial-Subgradient-Method} with $\alpha_i = \dfrac{\epsilon}{2\norm{\zeta_i}^2}$. Then some iteration
$$i \leq \left\lceil \frac{4}{3}  \frac{  \mathrm{dist}(x_0, \hat{X})^2}{R^2} \frac{1}{\epsilon^2}  \right\rceil$$
either identifies that $f$ is unbounded on the ray $\{t \tilde x_{i+1} \mid t>0\}$ or has 
$$\frac{f(x_i) - \hat{f}}{0 -\hat{f}} \leq \epsilon.$$
\end{theorem}

\begin{theorem} \label{thm:Convergence-When-Optimal-Known}
Suppose $f^*$ is finite and attained at some set of points $X^*$. Let $\{x_i\}$ be the sequence generated by Algorithm \ref{alg:Radial-Subgradient-Method} with $\alpha_i = \dfrac{z_{i} - f^*}{0-f^*} \dfrac{1}{\norm{\zeta_i}^2}$. Then for any $\epsilon>0$, some iteration
$$i \leq \left\lceil \frac{  \mathrm{dist}(x_0, X^*)^2}{R^2} \frac{1}{\epsilon^2}  \right\rceil \hskip0.4cm \text{ has } \hskip0.4cm \frac{f(x_i) - f^*}{0 - f^*} \leq \epsilon.$$
\end{theorem}

These convergence bounds are remarkably simple and should have fairly small constants in practice. 
In Section \ref{subsec:Renegars-Results}, we compare our results to those of Renegar in \cite{Renegar2016}. Then in Section \ref{sec:Preliminaries}, we establish a number of relevant properties of our radial reformulation, which follow from the equivalent structures in Renegar's framework. Finally in Section \ref{sec:Analysis-Of-Convergence}, we prove our main results on the convergence of the Radial Subgradient Method.

\subsection{Comparison to Previous Results} \label{subsec:Renegars-Results}
Renegar presents two subgradient algorithms (referred to as Algorithms A and B in \cite{Renegar2016}), which we paraphrase below in terms of our function-oriented transformation as Algorithms \ref{alg:Renegar-Unknown-Optimal} and \ref{alg:Renegar-Known-Optimal}. Algorithms \ref{alg:Renegar-Unknown-Optimal} assumes a target accuracy $\epsilon>0$ is given as input. Algorithm \ref{alg:Renegar-Known-Optimal} assumes the optimal objective value $f^*$ is given as input.

Renegar's analysis of these algorithms requires two additional assumptions on $f$ beyond the basic assumptions we make of lower-semicontinuity and convexity. In particular, they further assume that the minimum objective value is finite and attained at some point, and that the set of optimal solutions is bounded.

\begin{algorithm}
\caption{Function-oriented statement of Algorithm A in \cite{Renegar2016}}
\label{alg:Renegar-Unknown-Optimal}
\begin{algorithmic}[1]
\STATE $x_0=\vec 0$, $z_{0}=f(x_0)$
\FOR{$i = 0, 1, \dots$}
   \STATE Select $\zeta_i \in \partial \gamma_{z_{i}}(x_i)$
   \STATE $x_{i+1} = x_i - \alpha_i \zeta_i$, where $\alpha_i=\dfrac{\epsilon}{2\norm{\zeta_i}^2}$
   \STATE $z_{i+1} = z_i$
   \STATE {\bf if } $\gamma_{z_{i+1}}(x_{i+1}) \leq 3/4$ {\bf then } $(x_{i+1},z_{i+1}) := \dfrac{1}{\gamma_{z_{i+1}}(x_{i+1})}(x_{i+1},z_{i+1})$
\ENDFOR
\end{algorithmic}
\end{algorithm}
\begin{algorithm}
\caption{Function-oriented statement of Algorithm B in \cite{Renegar2016}}
\label{alg:Renegar-Known-Optimal}
\begin{algorithmic}[1]
\STATE $x_0=\vec 0$, $z=f^*$
\FOR{$i = 0, 1, \dots$}
   \STATE Select $\zeta_i \in \partial \gamma_{z}(x_i)$
   \STATE $x_{i+1} = x_i - \alpha_i \zeta_i$, where $\alpha_i=\dfrac{\gamma_z(x_i)-1}{\norm{\zeta_i}^2}$
\ENDFOR
\end{algorithmic}
\end{algorithm}

Notice that Algorithm \ref{alg:Renegar-Unknown-Optimal} only does a radial update when the threshold $\gamma_{z}(x) \leq 3/4$ is met and Algorithm \ref{alg:Renegar-Known-Optimal} never does radial updates. This contrasts with our Radial Subgradient Method which does updates every iteration.
Recently in \cite{Freund2016}, Freund and Lu presented an interesting approach to first-order optimization with similarities to Renegar's approach. Their method utilizes a similar thresholding condition for doing periodic updates.

As previously mentioned, Renegar's algorithms completely avoid computing orthogonal projections. As a result, these algorithms can have substantially lower per iteration cost than traditional subgradient methods reliant on orthogonal projections.

Define $D$ to be the diameter of the sublevel set $\{x \mid f(x)\leq f(\vec 0)\}$. Like $R$, this scalar only appears in convergence bounds, but is never assumed to be known. Note that the set of optimal solutions must be bounded for $D$ to be finite. Then Renegar proves the following convergence bounds.
 
\begin{theorem}[Renegar \cite{Renegar2016}, Theorem 1.1] \label{thm:Renegar-Convergence-When-Optimal-Unknown}
Consider any $\epsilon>0$. If $\{x_i\}$ is the sequence generated by Algorithm \ref{alg:Renegar-Unknown-Optimal}, then some iteration
\begin{equation*}
i \leq \left\lceil 8 \left(\frac{D}{R}\right)^2 \left(\frac{1}{\epsilon^2} +\frac{1}{\epsilon}\log_{4/3}\left(1 + \frac{D}{R}\right)\right) \right\rceil \hskip0.4cm \text{ has } \hskip0.4cm \frac{f(x_i/\gamma_{z_i}(x_i)) - f^*}{0 -f^*}\leq \epsilon.
\end{equation*}
\end{theorem}
\begin{theorem}[Renegar \cite{Renegar2016}, Theorem 1.2] \label{thm:Renegar-Convergence-When-Optimal-Known}
If $\{x_i\}$ is the sequence generated by Algorithm \ref{alg:Renegar-Known-Optimal}, then for any $0<\epsilon<1$, some iteration
\begin{equation*}
i \leq \left\lceil 4 \left(\frac{D}{R}\right)^2 \left(\frac{4}{3}\left(\frac{1-\epsilon}{\epsilon}\right)^2 +4\frac{1-\epsilon}{\epsilon} + \log_2\left(\frac{1-\epsilon}{\epsilon}\right) +\log_2\left(\frac{D}{R}\right)+1\right) \right\rceil
\end{equation*}
\begin{equation*}
\text{has }\hskip0.2cm \frac{f(x_i/\gamma_{z}(x_i)) - f^*}{0 -f^*} \leq \epsilon.
\end{equation*}
\end{theorem}

These bounds are remarkable in that they were the first to attain the same rate of growth with respect to $\epsilon$ as traditional methods that assume Lipschitz continuity. However, the constants involved could be very large. In particular, the value of $D$ can be enormous if a small perturbation of the problem would have an unbounded set of optimal solutions. Note if the set of optimal solutions $X^*$ is nonempty, then $D > \mathrm{dist}(x_0,X^*)$. As a result, our convergence bounds in Theorems \ref{thm:Convergence-When-Optimal-Unknown} and \ref{thm:Convergence-When-Optimal-Known} strictly improve upon those of Theorems \ref{thm:Renegar-Convergence-When-Optimal-Unknown} and \ref{thm:Renegar-Convergence-When-Optimal-Known}.

\section{Preliminaries} \label{sec:Preliminaries}

Renegar's framework begins by converting the problem of minimizing $f$ into conic form. Let $\mathcal{K}$ be the closure of $\{(xs,s,ts) \mid s > 0, (x,t)\in\mathrm{epi }f\}$, which is the conic extension of the epigraph of $f$. Note that the restriction of $\mathcal{K}$ to $s=1$ is exactly the epigraph of $f$. Then Renegar considers the following equivalent conic form problem
\begin{equation} \label{eq:Conic-Form-Of-Convex-Function}
\begin{cases}
	\min\    & t \\
    \text{s.t. } & s=1, (x,s,t) \in \mathcal{K}.
\end{cases}
\end{equation}

Observe that $(\vec 0, 1, 0)$ lies in the interior of $\mathcal{K}$. Then Renegar defines the following function, which lies at the heart of the framework. Although this function was originally developed for any conic program, we state it specifically in terms the above program.
$$\lambda(x,s,t) = \inf\{\lambda \mid (x,s,t) - \lambda(\vec 0,1,0) \not\in\mathcal{K}\}.$$

At this point, it is easy to show the connection between this function when restricted to $t=z$ and our radial reformulation of level $z$.
\begin{lemma}\label{lem:Lambda-To-Radial-Reformulation}
For any $x\in\R^n$ and $z\in\R$, $\gamma_z(x) = 1-\lambda(x,1,z)$.
\end{lemma}
\begin{proof}
Follows directly from the definitions of $\mathcal{K}$ and $\gamma_z(\cdot)$:
\begin{align*}
\lambda(x,1,z) & =  \sup\left\{\lambda \mid (x,1,z) - \lambda(\vec 0,1,0) \in\mathcal{K}\right\}\\
& = \sup\left\{\lambda \mid 1-\lambda > 0, (\frac{x}{1-\lambda},\frac{z}{1-\lambda}) \in\mathrm{epi }f\right\}\\
& = \sup\left\{\lambda < 1 \mid f\left(\frac{x}{1-\lambda}\right) \leq \frac{z}{1-\lambda} \right\}\\
& = 1 - \inf\left\{(1 - \lambda) > 0 \mid (1-\lambda)f\left(\frac{x}{1-\lambda}\right) \leq z \right\}\\
& = 1 - \gamma_z(x).  
\end{align*}
\end{proof}

As a consequence of Propositions 2.1 and 3.2 of \cite{Renegar2016}, we know that the radial reformulation is both convex and Lipschitz.
\begin{proposition} \label{prop:Lipschitz-Of-Radial-Reformulation}
For any $z\in\R$, the radial reformulation of level $z$, $\gamma_z(\cdot)$, is convex and Lipschitz with constant $1/R$ (independent of the level $z$).
\end{proposition}

We now see that the radial reformulation is notably more well-behaved than the original function $f$. However, to justify it as a meaningful proxy for the original function in an optimization setting, we need to relate their minimum values. In the following proposition, we establish such a connection between $f$ and any radial reformulation with a negative level.

Note that $f^p$ is strictly decreasing in $\gamma$. To see this, observe that, for any $x\in\R^n$, all sufficiently large $\gamma$ have $f(x/\gamma) < f(\vec 0)/2<0$. Then it follows that $\lim\limits_{\gamma\to \infty} f^p(x,\gamma) \leq \lim\limits_{\gamma\to \infty} \gamma f(\vec 0)/2 = -\infty$. Since perspective functions are convex (see \cite{Boyd-ConvexOptimization} for an elementary proof of this fact), we conclude that $f^p$ is strictly decreasing in $\gamma$.

\begin{proposition} \label{prop:Minimum-Of-Radial-Reformulation}
For any $z<0$, the minimum value of $\gamma_z(\cdot)$ is $z/f^*$ (where $z/-\infty := 0$).
Further, if $\gamma_z(x)=0$, then $\lim\limits_{t\to\infty}f(t x) = -\infty$. 
\end{proposition}
\begin{proof}
First we show that this minimum value lower bounds $\gamma_z(x)$ for all $x\in\R^n$. Since $ f\left((f^*/z)x\right) \geq f^*$, we know that $\gamma=z/f^*$ has $f^p(x,\gamma) \geq z$. Thus $\gamma_z(x)\geq z/f^*$ since $f^p$ is strictly decreasing in $\gamma$.

Consider any sequence $\{x_i\}$ with $f(x_i)<0$ and $\lim\limits_{i\to \infty}f(x_i) = f^*$. Observe that 
$$f^p\left( \frac{z}{f(x_i)}x_i, \frac{z}{f(x_i)}\right) = \frac{z}{f(x_i)}f(x_i) = z.$$
Since $f^p$ is strictly decreasing in $\gamma$, we have $\gamma_z\left((z/f(x_i))x_i\right) = z/f(x_i)$. It follows that $\lim\limits_{i\to\infty} \gamma_z\left((z/f(x_i))x_i\right) = z/f^*$. Then our lower bound is indeed the minimum value of the radial reformulation.
 
Our second observation follows from the definition of the radial reformulation (using the change of variables $t = 1/\gamma$):
\begin{align*}
\gamma_z(x)=0 & \Leftrightarrow \inf\{\gamma>0 \mid f(x/\gamma) < z/\gamma \}=0\\
& \Leftrightarrow \sup\{t >0 \mid f(t x) < t z\}=\infty\\
& \Rightarrow \lim_{t\to\infty}f(t x) = -\infty. 
\end{align*}
\end{proof}

Finally, we give a characterization of the subgradients of our radial reformulation. Although this description is not necessary for our analysis, it helps establish the practicality of the Radial Subgradient Method. We see that the subgradients of $\gamma_z(\cdot)$ can be computed easily from normal vectors of the epigraph of the original function. This result follows as a direct consequence of Proposition 7.1 of \cite{Renegar2016}.

\begin{proposition} \label{prop:Subgradients-Of-Radial}
For any $z<0$, the subgradients of the radial reformulation are given by 
$$\partial\gamma_z(x) = \left\{ \frac{\gamma_z(x)}{\langle \zeta, x\rangle + \delta z}\zeta \mid (\vec 0,0)\neq(\zeta, \delta) \in N_{\mathrm{epi } f}\left(\frac{x}{\gamma_z(x)}, \frac{z}{\gamma_z(x)}\right)\right\}.$$
\end{proposition}

\section{Analysis of Convergence} \label{sec:Analysis-Of-Convergence}

First, we observe the following two properties hold at each iteration of our algorithm.
\begin{lemma} \label{lem:Identity-For-Radial-Reformulation}
At any iteration $k\geq0$ of Algorithm \ref{alg:Radial-Subgradient-Method}, $\gamma_{z_{k}}(x_k)=1$.
\end{lemma}
\begin{proof}
When $k=0$, we have $f^p(x_0,1) = z_{0}$, and so the result follows from $f^p$ being strictly decreasing in $\gamma$. The general case follows from simple algebraic manipulation:
\begin{align*}
\gamma_{z_{k+1}}(x_{k+1}) & = \inf\left\{\gamma > 0 \mid \gamma f(x_{k+1}/\gamma) \leq z_{k+1}\right\}\\
& = \inf\left\{\gamma > 0 \mid \gamma f\left(\frac{\tilde x_{k+1}}{\gamma_{z_{k}}(\tilde x_{k+1})\gamma}\right) \leq z_{k+1}\right\}\\
& = \inf\left\{\gamma > 0 \mid \gamma_{z_{k}}(\tilde x_{k+1}) \gamma f\left(\frac{\tilde x_{k+1}}{\gamma_{z_{k}}(\tilde x_{k+1})\gamma}\right) \leq z_{k}\right\}\\
& = \frac{1}{\gamma_{z_{k}}(\tilde x_{k+1})}\gamma_{z_{k}}(\tilde x_{k+1}) = 1. 
\end{align*}
\end{proof}
\begin{lemma} \label{lem:Ordering-Of-Values}
At any iteration $k\geq0$ of Algorithm \ref{alg:Radial-Subgradient-Method}, the following ordering holds
$$f^* \leq f(x_k) \leq z_{k} < 0.$$
\end{lemma}
\begin{proof}
The first inequality follows from $f^*$ being the minimum value of $f$.
The second inequality is trivially true when $k=0$. From the lower-semicontinuity of $f^p(x,\gamma)$, we know $f^p(\tilde x_{k+1}, \gamma_{z_{k}}(\tilde x_{k+1})) \leq z_{k}$. Then we have the second inequality in general since $f(x_{k+1}) = f(\tilde x_{k+1}/\gamma_{z_{k}}(\tilde x_{k+1})) \leq z_{k}/\gamma_{z_{k}}(\tilde x_{k+1}) = z_{k+1}$.
The third inequality follows inductively since $z_{0}<0$ and $z_{k+1} = z_{k}/\gamma_{z_{k}}(\tilde x_{k+1})<0$.
\end{proof}

The traditional analysis of subgradient descent, assuming Lipschitz continuity, is based on an elementary inequality, which is stated in the following lemma. 
\begin{lemma} \label{lem:Traditional-Subgradient-Method-Inequality}
Consider any convex function $g:\R^n\to\R$, $x, y\in\R^n$, and $\zeta\in\partial g(x)$. Then for any $\alpha>0$,
$$\norm{(x - \alpha \zeta) - y}^2 \leq \norm{x-y}^2 - 2\alpha (g(x) - g(y)) + \alpha^2\norm{\zeta}^2.$$
\end{lemma}
\begin{proof}
Follows directly from applying the subgradient inequality, $g(y)\geq g(x) + \langle\zeta, y-x \rangle$:
\begin{align*}
\norm{(x - \alpha \zeta) - y}^2 &= \norm{x-y}^2 - 2\alpha \langle\zeta, x-y \rangle + \alpha^2\norm{\zeta}^2 \\
& \leq \norm{x-y}^2 - 2\alpha (g(x) - g(y)) + \alpha^2\norm{\zeta}^2. 
\end{align*}
\end{proof}

The core of proving the traditional subgradient descent bounds is to inductively apply this lemma at each iteration. However, such an approach cannot be applied directly to Algorithm \ref{alg:Radial-Subgradient-Method} since the iterates are rescaled every iteration and the underlying function changes every iteration. The key to proving our convergence bounds is setting up a modified inequality that can be applied inductively. This is done in the following lemma.

\begin{lemma}\label{lem:Modified-Subgradient-Inequality}
Consider any $y$ with $f(y)<0$. Then at any iteration $k\geq0$ of Algorithm \ref{alg:Radial-Subgradient-Method},
$$\norm{\frac{f(y)}{z_{k+1}}x_{k+1} - y}^2 \leq \norm{\frac{f(y)}{z_{k}}x_{k} - y}^2 - 2\alpha_k\frac{f(y)}{z_{k}}\frac{z_{k}-f(y)}{0-z_{k}} + \alpha_k^2 \left(\frac{f(y)}{z_{k}}\right)^2\norm{\zeta_k}^2.$$
\end{lemma}
\begin{proof}
Applying Lemma \ref{lem:Traditional-Subgradient-Method-Inequality} on $\gamma_{z_{k}}(\cdot)$ with $x_k$ and $\dfrac{z_{k}}{f(y)}y$ implies
\begin{equation}
\norm{\tilde x_{k+1} - \frac{z_{k}}{f(y)}y}^2 \leq \norm{x_{k} - \frac{z_{k}}{f(y)}y}^2 - 2\alpha_k\left(\gamma_{z_{k}}(x_k) - \gamma_{z_{k}}\left(\frac{z_{k}}{f(y)}y\right)\right) + \alpha_k^2\norm{\zeta_k}^2.
\end{equation}
The value of $\gamma_{z_{k}}\left(\dfrac{z_{k}}{f(y)}y\right)$ can be derived directly. Observe that
$$f^p\left(\frac{z_{k}}{f(y)}y, \frac{z_{k}}{f(y)}\right) = \frac{z_{k}}{f(y)} f(y) =z_{k}.$$
Then $\gamma_{z_{k}}\left(\dfrac{z_{k}}{f(y)}y\right) = \dfrac{z_{k}}{f(y)}$ since $f^p$ is strictly decreasing in $\gamma$.
Combining this with Lemma \ref{lem:Identity-For-Radial-Reformulation} allows us to restate our inequality as
\begin{equation}
\norm{\tilde x_{k+1} - \frac{z_{k}}{f(y)}y}^2 \leq \norm{x_{k} - \frac{z_{k}}{f(y)}y}^2 -2\alpha_k\left(1 - \frac{z_{k}}{f(y)}\right) + \alpha_k^2\norm{\zeta_k}^2.
\end{equation}
Multiplying through by $(f(y)/z_{k})^2$ yields
\begin{equation}
\norm{\frac{f(y)}{z_{k}}\tilde x_{k+1} - y}^2 \leq \norm{\frac{f(y)}{z_{k}}x_{k} - y}^2 -2\alpha_k\frac{f(y)}{z_{k}}\frac{z_{k} - f(y)}{0-z_{k}} + \alpha_k^2\left(\frac{f(y)}{z_{k}}\right)^2\norm{\zeta_k}^2.
\end{equation}
Noting that $\tilde x_{k+1} = \dfrac{z_{k}}{z_{k+1}}x_{k+1}$ completes the proof.
\end{proof}

\subsection{Proof of Theorem \ref{thm:Convergence-In-General}} \label{subsec:Proof-In-General}
We assume that for all $i\leq k$, we have $\gamma_{z_{i}}(\tilde x_{i+1})>0$ (otherwise the theorem immediately holds by Proposition \ref{prop:Minimum-Of-Radial-Reformulation}). Then the first $k$ iterates of Algorithm \ref{alg:Radial-Subgradient-Method} are well-defined.
Consider any $\hat{x}\in\hat{X}$. Inductively applying Lemma \ref{lem:Modified-Subgradient-Inequality} with $y=\hat{x}$ produces
\begin{equation}
\norm{\frac{\hat{f}}{z_{k+1}}x_{k+1} - \hat{x}}^2 \leq \norm{\frac{\hat{f}}{z_{0}}x_{0} - \hat{x}}^2 -\sum^k_{i=0}\left( 2\alpha_i\frac{\hat{f}}{z_{i}}\frac{z_{i}-\hat{f}}{0-z_{i}} - \alpha_i^2 \left(\frac{\hat{f}}{z_{i}}\right)^2 \norm{\zeta_i}^2\right).
\end{equation}
Noting that $x_0 = \vec 0$, this implies
\begin{equation}
2\sum^k_{i=0}\alpha_i\frac{\hat{f}}{z_{i}}\frac{z_{i}-\hat{f}}{0-z_{i}}  \leq \norm{\hat{x}}^2 + \sum^k_{i=0} \alpha_i^2 \left(\frac{\hat{f}}{z_{i}}\right)^2 \norm{\zeta_i}^2.
\end{equation}
Minimizing over all $\hat{x}\in\hat{X}$, we have 
\begin{equation} \label{eq:Summed-Up-Inequality}
2\sum^k_{i=0}\alpha_i\frac{\hat{f}}{z_{i}}\frac{z_{i}-\hat{f}}{0-z_{i}}  \leq \textrm{dist}(x_0, \hat{X})^2 + \sum^k_{i=0} \alpha_i^2 \left(\frac{\hat{f}}{z_{i}}\right)^2 \norm{\zeta_i}^2.
\end{equation}
From Proposition \ref{prop:Lipschitz-Of-Radial-Reformulation}, we have $\norm{\zeta_i}\leq 1/R$. Then rearrangement of this inequality gives
\begin{equation}
\min_{i\leq k}\left\{\frac{z_{i}-\hat{f}}{0-z_{i}}\right\} \leq \frac{\textrm{dist}(x_0, \hat{X})^2 + \frac{1}{R^2}\sum^k_{i=0} \alpha_i^2 \left(\frac{\hat{f}}{z_{i}}\right)^2}{2\sum^k_{i=0}\alpha_i\frac{\hat{f}}{z_{i}}}.
\end{equation}
Then Theorem \ref{thm:Convergence-In-General} follows from the fact that $ f(x_i) \leq z_{i} < 0$, as shown in Lemma \ref{lem:Ordering-Of-Values}.

\subsection{Proof of Corollary \ref{cor:Convergence-Square-Summable}} \label{subsec:Proof-Square-Summable-Corrollary}
Corollary \ref{cor:Convergence-Square-Summable} follows directly from Theorem \ref{thm:Convergence-In-General}. By Proposition \ref{prop:Minimum-Of-Radial-Reformulation}, the algorithm will correctly report unbounded objective if it ever encounters $\gamma_{z_{i}}(\tilde x_{i+1})=0$ (and thus the corollary holds). So we assume this never occurs, which implies the sequence of iterates $\{x_i\}$ is well-defined.
By selecting $\alpha_i=-z_{i}\beta_i$, the upper bound in Theorem \ref{thm:Convergence-In-General} converges to zero:
\begin{equation*}
\lim_{k\to\infty} \frac{\textrm{dist}(x_0, \hat{X})^2 + \frac{1}{R^2}\sum^k_{i=0} \alpha_i^2 \left(\frac{\hat{f}}{z_{i}}\right)^2}{2\sum^k_{i=0}\alpha_i\frac{\hat{f}}{z_{i}}} = \lim_{k\to\infty} \frac{\textrm{dist}(x_0, \hat{X})^2 + \hat{f}^2\frac{1}{R^2}\sum^k_{i=0} \beta_i^2}{-2\hat{f}\sum^k_{i=0}\beta_i} = 0.
\end{equation*}
Then we have the following convergence result
$$ \lim_{k\to\infty}\min_{i\leq k}\left\{\frac{f(x_i)-\hat{f}}{0-f(x_i)}\right\}\leq 0.$$
This implies $\lim\limits_{k\to\infty}\min\limits_{i\leq k}\left\{f(x_i)\right\}\leq \hat{f}$. Considering a sequence of $\hat{f}$ approaching $f^*$ gives the desired result.

\subsection{Proof of Theorem \ref{thm:Convergence-When-Optimal-Unknown}} \label{subsec:Proof-When-Optimal-Unknown} 
Let $k = \left\lceil \frac{4}{3}  \frac{  \mathrm{dist}(x_0, \hat{X})^2}{R^2} \frac{1}{\epsilon^2}  \right\rceil$. We assume $\gamma_{z_{i}}(\tilde x_{i+1})>0$ for all $i\leq k$ (otherwise the theorem immediately holds by Proposition \ref{prop:Minimum-Of-Radial-Reformulation}). Then the first $k$ iterates of Algorithm \ref{alg:Radial-Subgradient-Method} are well-defined.
Combining (\ref{eq:Summed-Up-Inequality}) with our choice of step size $\alpha_i = \dfrac{\epsilon}{2\norm{\zeta_i}^2}$ yields
\begin{equation}
\sum^k_{i=0}\frac{\epsilon}{\norm{\zeta_i}^2}\frac{\hat{f}}{z_{i}}\frac{z_{i}-\hat{f}}{0-z_{i}}  \leq \textrm{dist}(x_0, \hat{X})^2 + \sum^k_{i=0}\left(\frac{\hat{f}}{z_{i}}\right)^2 \dfrac{\epsilon^2}{4\norm{\zeta_i}^2}
\end{equation}
\begin{equation}
\implies \sum^k_{i=0}\frac{\epsilon}{\norm{\zeta_i}^2}\left(\frac{\hat{f}}{z_{i}}\right)^2\left(\frac{z_{i}-\hat{f}}{0-\hat{f}} - \frac{\epsilon}{4}\right) \leq \mathrm{dist}(x_0, \hat{X})^2
\end{equation}
\begin{equation}
\implies \epsilon (k+1) \min_{i\leq k}\left\{\frac{1}{\norm{\zeta_i}^2}\left(\frac{\hat{f}}{z_{i}}\right)^2\left(\frac{z_{i}-\hat{f}}{0-\hat{f}} - \frac{\epsilon}{4}\right)\right\} \leq \mathrm{dist}(x_0, \hat{X})^2.
\end{equation}
If any $i\leq k$ has $\hat{f}>z_i$, the theorem holds (as this would imply $f(x_i) - \hat{f} \leq z_{i} -\hat{f}<0$). So we now assume $\hat{f} / z_{i} \geq 1$. Then, noting that $\norm{\zeta_i}\leq 1/R$, we can simplify our inequality to
\begin{equation}
\min_{i\leq k}\left\{\frac{z_{i}-\hat{f}}{0-\hat{f}} - \frac{\epsilon}{4}\right\} \leq \frac{\mathrm{dist}(x_0, \hat{X})^2}{\epsilon R^2 (k+1)}.
\end{equation}
Since $f(x_{i}) \leq z_{i}$ from Lemma \ref{lem:Ordering-Of-Values}, we have the following, completing our proof of Theorem \ref{thm:Convergence-When-Optimal-Unknown},
\begin{equation}
\min_{i\leq k}\left\{ \frac{f(x_i) - \hat{f}}{0-\hat{f}} \right\} \leq \frac{\mathrm{dist}(x_0, \hat{X})^2}{\epsilon R^2(k+1)} + \frac{\epsilon}{4}.
\end{equation}

\subsection{Proof of Theorem \ref{thm:Convergence-When-Optimal-Known}} \label{subsec:Proof-When-Optimal-Known}
Since we assume $f^*$ is finite, we know all $\gamma_{z_{i}}(\tilde x_{i+1})\geq z_{i}/f^* >0$ (by Proposition \ref{prop:Minimum-Of-Radial-Reformulation}). Thus the sequence of iterates $\{x_i\}$ is well-defined.
Consider any $k\geq 0$. Then taking (\ref{eq:Summed-Up-Inequality}) with $\hat{f}=f^*$ gives
\begin{equation}
2\sum^k_{i=0}\alpha_i\frac{f^*}{z_{i}}\frac{z_{i}-f^*}{0-z_{i}} \leq \mathrm{dist}(x_0, X^*)^2 + \sum^k_{i=0} \alpha_i^2 \left(\frac{f^*}{z_{i}}\right)^2 \norm{\zeta_i}^2.
\end{equation}
Combining this with our choice of step size $\alpha_i = \dfrac{z_{i} - f^*}{0-f^*} \dfrac{1}{\norm{\zeta_i}^2}$ yields
\begin{equation}
2\sum^k_{i=0} \frac{1}{\norm{\zeta_i}^2}\left(\frac{z_{i}-f^*}{0-z_{i}}\right)^2 \leq \mathrm{dist}(x_0, X^*)^2 + \sum^k_{i=0} \frac{1}{\norm{\zeta_i}^2}\left(\frac{z_{i}-f^*}{0-z_{i}}\right)^2
\end{equation}
\begin{equation}
\implies \sum^k_{i=0} \frac{1}{\norm{\zeta_i}^2}\left(\frac{z_{i}-f^*}{0-z_{i}}\right)^2 \leq \mathrm{dist}(x_0, X^*)^2
\end{equation}
\begin{equation}
\implies (k+1)\min_{i\leq k}\left\{\left( \frac{1}{\norm{\zeta_i}^2}\frac{z_{i}-f^*}{0-z_{i}}\right)^2\right\} \leq \mathrm{dist}(x_0, X^*)^2.
\end{equation}
From Proposition \ref{prop:Lipschitz-Of-Radial-Reformulation}, we know $\norm{\zeta_i} \leq 1/R$, and thus
\begin{equation}
\min_{i\leq k}\left\{\left(\frac{z_{i}-f^*}{0-z_{i}}\right)^2\right\} \leq \frac{\mathrm{dist}(x_0, X^*)^2}{R^2(k+1)}.
\end{equation}
From Lemma \ref{lem:Ordering-Of-Values}, we have $f^* \leq f(x_i)\leq z_{i}$, which completes our proof by implying
\begin{equation}
\min_{i\leq k}\left\{\frac{f(x_i) - f^*}{0-f^*}\right\} \leq \frac{\mathrm{dist}(x_0, X^*)}{R\sqrt{k+1}}.
\end{equation}

{\bf Acknowledgments.} The author wishes to express his deepest gratitude to Jim Renegar for numerous insightful discussions helping motivate this work, and for advising on both the presentation and positioning of this paper.

\bibliographystyle{siamplain}
\bibliography{references}

\end{document}